\documentclass[11pt]{amsart}
\usepackage{hyperref}
\usepackage{anysize} \marginsize{1.3in}{1.3in}{1in}{1in}
\usepackage{amsfonts}

\usepackage[all,cmtip]{xy}
\usepackage{amsmath}
\usepackage{amsthm}
\usepackage{amssymb}
\usepackage{mathrsfs}
\usepackage{enumitem}
\usepackage{tikz-cd}
\usepackage{tikz}
\usepackage{comment}

\usepackage[normalem]{ulem}

\newtheorem{thm}{Theorem}[section]

\newtheorem{prop}[thm]{Proposition}
\newtheorem{lemma}[thm]{Lemma}

\newtheorem{lem}[thm]{Lemma}
\newtheorem{cor}[thm]{Corollary}

\theoremstyle{definition}
\newtheorem{defn}[thm]{Definition}
\newtheorem*{definition*}         {Definition}

\newtheorem{eg}[thm]{Example}
\theoremstyle{remark}
\newtheorem*{claim}{Claim}
\newtheorem{remark}[thm]{Remark}

\renewcommand{\O}{\mathcal{O}}
\newcommand*{\Q}{\mathbb{Q}}
\newcommand*{\HH}{\mathbb{H}}
\newcommand*{\Abb}{\mathbb{A}}

\newcommand*{\A}{\mathcal{A}}

\newcommand*{\Z}{\mathbb{Z}}
\newcommand*{\G}{\mathbf{G}}

\newcommand*{\R}{\mathbb{R}}

\newcommand*{\C}{\mathbb{C}}

\newcommand*{\mat}[4]{\left(\begin{smallmatrix} #1 & #2\\#3 & #4 \end{smallmatrix}\right)}

\newcommand*{\PP}{\mathbb{P}}

\newcommand*{\Hom}{\textrm{Hom}}

\newcommand*{\ra}{\rightarrow}

\newcommand*{\ol}{\overline}

\def\alg{{\mathrm{alg}}}

\newcommand{\cD}{\mathcal{D}}

\newcommand{\cH}{\mathcal{H}}
\newcommand{\sH}{\mathscr{H}}
\newcommand{\cE}{\mathcal{E}}
\newcommand{\cM}{\mathcal{M}}
\newcommand{\cB}{\mathcal{B}}
\newcommand{\cP}{\mathcal{P}}
\newcommand{\cF}{\mathcal{F}}
\newcommand{\cV}{\mathcal{V}}

\def\sHom{\operatorname{\mathscr{H}\hspace{-.35em}\mathit{om}}}

\renewcommand{\phi}{\varphi}

\def\Hom{\operatorname{Hom}}

\def\Ext{\operatorname{Ext}}

\def\pt{\operatorname{pt}}

\def\gr{\operatorname{gr}}

\def\Aut{\operatorname{Aut}}

\def\Hdg{\operatorname{Hdg}}
\def\MHM{\operatorname{MHM}}
\def\rat{\operatorname{rat}}

\def\reduced{\mathrm{red}}

\newcommand{\df}{\mathrm{def}}
\newcommand{\an}{\mathrm{an}}

\renewcommand{\bar}[1]{\overline{#1}}

\usepackage{amscd,amssymb,amsmath}
\usepackage{color}

\title{Quasiprojectivity of images of mixed period maps} 

 \author[B. Bakker]{Benjamin Bakker}
\address{\noindent B. Bakker:  Dept. of Mathematics, University of Georgia, Athens, USA.}
\email{bakker.uga@gmail.com}

\author[Y. Brunebarbe]{Yohan Brunebarbe}
\address{\noindent Y. Brunebarbe:  Dept. of Mathematics, Univ. Bordeaux, Talence, France.}
\email{yohan.brunebarbe@math.u-bordeaux.fr}

\author[J. Tsimerman]{Jacob Tsimerman}
\address{\noindent J. Tsimerman:  Dept. of Mathematics, University of Toronto, Toronto, Canada.}
\email{jacobt@math.toronto.edu}

\begin{document}
\begin{abstract} We prove a mixed version of a conjecture of Griffiths:  that the closure of the image of any admissible mixed period map is quasiprojective, with a natural ample bundle. Specifically, we consider the map from the image of the mixed period map to the image of the period map of the associated graded.  On the one hand, we show in a precise manner that the parts of this map parametrizing extension data of non-adjacent-weight pure Hodge structures are quasi-affine.  On the other hand, extensions of adjacent-weight pure polarized Hodge structures are parametrized by a compact complex torus (the intermediate Jacobian) equipped with a natural theta bundle which is ample in Griffiths transverse directions.

Our proof makes heavy use of o-minimality, and recent work with B. Klingler associating a $\R_{\an,\exp}$-definable structure to mixed period domains and admissible mixed period maps. \end{abstract}
\maketitle

\section{Introduction}

Let $X$ be an algebraic space and $\cM$ a moduli space of graded-polarized integral mixed Hodge structures, henceforth referred to as a period space.  There is a period space $\cD$ parametrizing the associated graded objects of the points in $\cM$ with a map $\cM\to \cD$, and to any period map $\phi:X\to \cM$ corresponding to a variation of graded-polarized integral mixed Hodge structures there is a period map $\gr\phi:X\to \cD$ for the associated graded variation.  By the results of \cite{bbt} we have a factorization
\[\begin{tikzcd}
X\arrow{rr}{\gr\phi}\arrow{dr}[swap]{g}&&\cD\\
&Z\arrow[ru,hook,swap,"\epsilon"]&
\end{tikzcd}\]
where $g$ is algebraic and dominant (meaning $\O_Z\to g_*\O_X$ is injective) and $\epsilon$ is a closed immersion.  Moreover, the Griffiths bundle on $Z$ is algebraic and ample, and in particular $Z$ is quasiprojective.

Our main result is to extend this picture to the mixed case.  Precisely, we show:
\begin{thm}\label{thm main}  Let $X$ be a separated algebraic space of finite type over $\C$ and $\phi:X\to \cM$ the period map associated to an admissible\footnote{\label{footnote}See Definition \ref{defn admissible} for the definition of admissibility over nonreduced bases.} variation of (graded-polarized integral) mixed Hodge structures.  Then there is a factorization\[\begin{tikzcd}
X\arrow{rr}{\phi}\arrow{dr}[swap]{f}&&\cM.\\
&Y\arrow[ru,hook,swap,"\iota"]&
\end{tikzcd}\]
where $f$ is dominant algebraic and $\iota$ is a closed immersion.  Moreover, the natural theta bundle on $Y$ is algebraic and relatively ample over the image $Z$ of the period map of the associated graded.  In particular, $Y$ is quasiprojective. 
\end{thm}
The period space $\cM$ is naturally a quotient $\Gamma\backslash M$ of a graded-polarized integral mixed period domain $M$ by an arithmetic group $\Gamma$, but the same result for the quotient\footnote{Such quotients are not good moduli spaces however as they do not in general have a tame geometry.} $\Gamma_\mathrm{mon}\backslash M$ by the image of the monodromy representation easily follows from Theorem \ref{thm main} (see Corollary \ref{cor mon image}).  

As in \cite{bbt}, we for instance obtain as a corollary the following:

\begin{cor}\label{introcoarse} Let $\mathcal{X}$ be a separated Deligne--Mumford stack of finite type over $\C$ admitting a quasi-finite admissible\textsuperscript{\emph{\ref{footnote}}} mixed period map.  Then the coarse moduli space of $\mathcal{X}$ is quasi-projective.
\end{cor}

The factorization statement in Theorem \ref{thm main} follows easily from \cite{bbkt} and \cite{bbt}, and the main content is the relative ampleness of the theta bundle.  This is especially interesting compared to the corresponding result in the pure case as the positivity does not stem from the negative curvature of $\cM$; indeed, the fibers of $\cM\to\cD$ are flat.  

The theta bundle is loosely constructed as follows (see section \ref{sect theta} for details).  For any polarized integral pure Hodge structure $V=(V_\Z,F^\bullet V,q_\Z)$ of weight $-1$, extensions of the form
\begin{equation}0\to V\to E\to\Z(0)\to0\label{eq one step}\end{equation}
are parametrized by the intermediate Jacobian
\[J(V)=V_\C/F^0V+V_\Z\]
which is a (compact) complex torus.  The polarization of $V$ endows $J(V)$ with a natural theta bundle which is positive in Griffiths transverse directions.  Now for a general variation of graded-polarized integral mixed Hodge structures $E$, we obtain from each variation $\gr^W_{[w-1,w]}E:=W_{w}E/W_{w-2}E$ an extension of the form \eqref{eq one step} via the natural pullback:
\[\begin{tikzcd}[column sep=small]
0\arrow[r]&\gr_{w-1}^WE\otimes (\gr_{w}^WE)^\vee\arrow[d,equal]\arrow{r}&E' \ar[r]\arrow{d}&\Z(0)\arrow{r}\arrow{d}&0\\
0\arrow{r}&\gr_{w-1}^WE\otimes (\gr_{w}^WE)^\vee\arrow{r}&\gr^W_{[w-1,w]}E \otimes (\gr_{w}^WE)^\vee\arrow{r}&\gr^W_{w}E\otimes (\gr_{w}^WE)^\vee\arrow{r}&0
\end{tikzcd}\]
The theta bundle of Theorem \ref{thm main} is then the product $\Theta:=\bigotimes_w \Theta_{[w-1,w]}$ of the theta bundles $\Theta_{[w-1,w]}$ associated to each of the $\gr^W_{[w-1,w]}E$.  In fact, it is easy to see that $\bigotimes_i \Theta_{[w-1,w]}^{a_w}$ is $f$-ample for any $a_w>0$.  

There are two main difficulties in establishing the relative ampleness of $\Theta$.  First, we must show $\Theta$ is algebraic.  This follows for $X$ smooth by work of Brosnan--Pearlstein \cite{bparch} and in general by definable GAGA \cite{bbt}.  We also give a new proof of the result of Brosnan--Pearlstain, see Remark \ref{rmk new proof}.  

Second, the theta bundle only accounts for the compact parts of the extension data, and the rest of the argument is devoted to showing that the remaining extension data is affine.  More precisely, there are period maps for which $Y\to Z$ has positive-dimensional fibers but for which all of the $\gr^W_{[w-1,w]}E$ are locally constant on the fibers, and in this case the theorem requires $\O_Y$ to be relatively ample---i.e., that $Y$ is quasiaffine over $Z$.  This ultimately relies on the geometry of mixed period spaces parametrizing extensions as in \eqref{eq one step} with $V$ of weight $\leq -2$ and our argument critically uses the work of Saito.

\subsection{Outline}  In \S\ref{sect alg} we prove the factorization part of Theorem \ref{thm main} and the algebraicity of the theta bundle.  In \S\ref{sect setup} we prove an ampleness criterion in terms of point separation by definable sections.  We also apply the work of Saito to prove some results on the local monodromy of the unipotent part of a variation of mixed Hodge structures.  In \S\ref{sect proof} we prove the quasiprojectivity part of Theorem \ref{thm main}.  

\subsection{Acknowledgements}  Y.B. would like to thank P. Brosnan for an interesting discussion related to the biextension bundle.  B.B. was partially supported by NSF grants DMS-1702149 and DMS-1848049.

\subsection{Notation}  Unless otherwise stated, by definable we mean definable in the o-minimal structure $\R_{\an,\exp}$.  All algebraic spaces are assumed to be separated and of finite type over $\C$; all definable analytic spaces (resp. analytic spaces) are complex definable analytic spaces (resp. complex analytic spaces).  We generally do not distinguish notationally between algebraic spaces, their associated definable analytic spaces, or their associated analytic spaces.

\section{Algebraicity of period maps and theta bundles}\label{sect alg}

Throughout, we use the following terminology.  Let $(V_\Z,W_\bullet V_\Q)$ be a free finite rank $\Z$-module with an increasing rational filtration.  We denote by $\gr^W_wV_\Z$ the $w$th graded object $\gr^W_wV_\Q$ with the integral structure induced by $V_\Z$. For each $w$ let a $(-1)^w$-symmetric form $q_w$ on $\gr^W_wV_\Z$ be given. There is an associated graded-polarized mixed period domain $M$ parametrizing graded-polarized mixed Hodge structures on $(V_\Z,W_\bullet V_\Q,q_\bullet)$.  By a graded-polarized mixed period space we mean the quotient $\cM=\Gamma\backslash M$ by an arithmetic subgroup $\Gamma\subset \G(\Z):=\Aut(V_\Z,W_\bullet V_\Q,q_\bullet)$.  We have that $\cM$ is naturally an $\R_{\alg}$-definable analytic space by \cite{bbkt}. When the weight filtration has one nonzero graded piece we refer to $\cM$ (resp. $M$) as a polarized pure period space (resp. domain), and usually denote it by $\cD$ (resp. $D$).  We also denote by $\check{M}$ the ``compact" dual of $M$---the space of filtrations $F^\bullet$ on $V_\C$ with fixed $\dim \gr^p_F\gr^W_wV_\C$ such that the induced filtration $F^\bullet \gr^W_wV_\C$ is $q_w$-isotropic---which is naturally a complex algebraic variety.  See for instance \cite{usui, pearlstein, bbkt} for background on mixed period spaces.

\subsection{Admissible period maps}

For a definable analytic space $X$, by a definable period map we mean a definable locally liftable map $\phi:X\to \cM$ which is tangent to the Griffiths transverse foliation of $\cM$ on the reduced\footnote{Note in particular that we do not require the nilpotent tangent directions to be Griffiths transverse, though it is not clear that this level of generality is useful:  variations coming from geometry will satisfy Griffiths transversality in the nilpotent directions as well.} regular locus.  A definable period map is equivalent to a variation of graded-polarized integral mixed Hodge structures, which consists of:  a filtered local system $(\cV_\Z,W_\bullet \cV_\Q,q_\bullet)$ locally modeled on $(V_\Z,W_\bullet V_\Q,q_\bullet )$ and a locally split filtration $F^\bullet$ of $\cV_\Z\otimes_{\Z}\O_X$ by definable coherent subsheaves which satisfies Griffiths transversality on the reduced regular locus and which is fiberwise a graded-polarized integral mixed Hodge structure.

We briefly recall the notion of admissible variations; see for instance \cite{kashiwara} for details.  Let $(\cV_\Z, W_\bullet \cV_\Q, F^\bullet)$ be a variation of graded-polarizable integral mixed Hodge structures on the punctured disk $\Delta^*$ with unipotent monodromy. Let $\bar V $ and $\bar W_\bullet \bar V$ denote the canonical extensions of $\cV_\Z \otimes_{\Z} \O_{\Delta^*}$ and $W_\bullet\cV_\Q \otimes_{\Q} \O_{\Delta^*}$ to $\Delta$ respectively, equipped with their logarithmic connections. Recall that the variation $(\cV_\Z, W_\bullet \cV_\Z, F^\bullet)$ is called pre-admissible if the following conditions hold:
\begin{enumerate}
\item The residue at the origin of the logarithmic connection on $\bar V$, which is an endomorphism of the fiber $\bar V_{|0}$ of $\bar V$ at the origin, admits a weight filtration relative to $\bar W_{\bullet}\bar V_{|0}$.
\item The Hodge filtration $F^\bullet$ extends to a locally split filtration $\bar F^\bullet$ of $\bar V $ such that $\gr_ {\bar F}^p  \gr_k^{\bar W} \bar V$ is locally-free for all $p$ and $k$.
\end{enumerate}

Given a Zariski-open subset $S$ in a reduced complex analytic space $\bar S$, we say that a graded-polarizable variation $(\cV_\Z, W_\bullet\cV_\Q,F^\bullet)$ on $S$ is admissible with respect to the inclusion $S \subset \bar S$ if for any holomorphic map $f : \Delta \to \bar S$ such that $f(\Delta^*) \subset S$ and $f^* \cV_\Z$ has unipotent monodromy, the pull-back variation on $\Delta^*$ is pre-admissible.  One easily verifies that a variation on $\Delta^*$ with unipotent monodromy which is pre-admissible is admissible with respect to the inclusion $\Delta^* \subset \Delta$.  Moreover, if a variation over a complex algebraic variety $S$ is admissible with respect to an algebraic compactification $\bar S$ of $S$, then it is admissible with respect to any other algebraic compactification of $S$.

\begin{defn}\label{defn admissible}  Let $X$ be an algebraic space. We say a period map $\phi:X\to \cM$ is \emph{admissible} if it is definable and the reduced map $\phi^{\reduced}:X^{\reduced}\to\cM$ is admissible. 
\end{defn}

See Corollary \ref{cor tfae} for some further discussion on the admissibility condition in the nilpotent directions.

\subsection{Properness of admissible period maps}

We will need an extension property for mixed period maps in Lemma \ref{lemma make proper} that is analogous to Griffiths' result in the pure case \cite[Theorem 9.5]{Giii}.  This is most likely known to experts, but we include a full argument for the reader's convenience.  \\

We first prove a criterion of properness for definable analytic maps analogous to the valuative criterion of properness for algebraic maps.

\begin{lemma}\label{valuative criterion of properness}
Let $X$ be an algebraic space, $\cM$ a definable analytic space and $\phi : X \to \cM$ a definable analytic map. Then the map $\phi$ is proper if, and only if, the following property holds:  a definable holomorphic map $v: \Delta^* \to X $ extends to $\Delta$ as soon as $\phi \circ v : \Delta^* \to \cM$ does.
\end{lemma}
\begin{proof}
Clearly we can assume that both $X$ and $\cM$ are reduced. Let $\bar X$ be an algebraic space compactifying $X$ and let $\tilde X$ denote the topological closure of $X$ in $\bar X \times \cM$. Then $\tilde X$ is definable and analytic by Bishop's theorem \cite[Theorem 3]{bishop}, as definable sets have locally bounded volume. Since $\bar X$ is proper the induced holomorphic map $\tilde  \phi : \tilde X \to \cM$ is proper, and the map $\phi : X \to \cM$ is proper exactly when $X= \tilde X$. Assume first that $\phi$ is not proper, so that there exists $\tilde x \in \tilde X - X$. Since $\tilde X - X$ is a closed analytic subset of $\tilde X$ (as it is the intersection of $(\bar X - X) \times \cM$ with $\tilde X$), there exists $v : \Delta \to \tilde X$ a definable holomorphic map such that $v(\Delta^\ast) \subset X$ and $v(0) = \tilde x$. Then the map $\phi \circ v : \Delta^* \to \cM$ does extend to $\Delta$ but $v$ does not. Conversely, let $v : \Delta^* \to X$ be a definable holomorphic map such that $\phi \circ v $ extends to $\Delta$. The induced definable holomorphic map $\Delta^* \to X \times \cM$ extends to $\Delta \to \bar X \times \cM$ and takes values in $\tilde X = X$, hence we are done.

\end{proof}

We now apply this criterion to our situation.  Let $X$ be a smooth algebraic space, $\phi:X\to \cM$ an admissible period map, and let $X\subset \bar X$ be a smooth compactification such that $\bar X \backslash X=\bigcup_i D_i$ is a normal crossing divisor.
Note that for any $i$ the local monodromy around $D_i$ is quasi-unipotent.  We may cover $\bar X$ by polydisks $P\cong \Delta^{n_P}$ such that $P\cap X\cong (\Delta^*)^{r_P}\times\Delta^{s_P}$.  For each polydisk $P$ we choose a basepoint $x_P\in P$, and let $N_i^P$ be the logarithm of the unipotent part of the local monodromy operator associated to the $D_i$ meeting $P$.  For each $P$ let $C_P$ be the cone generated by $\{N^P_i\}$.

\begin{lemma}\label{lemma proper cond}
The period map $\phi$ is proper if, and only if none of the cones $C_P$ contains $0$.
\end{lemma}

When $\cM$ is a pure period space, it follows from the strong version of the nilpotent orbit theorem \cite[Theorem 1.15]{hodgeasymp} that the cone $C_P$ contains $0$ only if one of the $N^P_i$'s is zero, in accordance with Griffiths' result. 
\begin{proof}
Let $v: \Delta^* \to X $ be a definable holomorphic map. Thanks to the nilpotent orbit theorem \cite[Proposition 4.3]{bbkt}, the composition $\phi \circ v : \Delta^* \to \cM$ extends to $\Delta$ exactly when the monodromy around $0$ is zero. On the other hand, after shrinking one can assume that $v: \Delta^* \to X $ takes values in one of the polydisk $P$. Then the logarithm of the unipotent part of the monodromy around $0$ is of the form $\sum_i a_i \cdot N_i^P$ for some non-negative integers $a_i$, and conversely every integral element of $C_p$ arises from a definable holomorphic map  $\Delta^* \to P $. Since the map   $v: \Delta^* \to X$ extends to $\Delta \to X$ exactly when all the $a_i$'s are zero, we conclude using Lemma \ref{valuative criterion of properness}.

\end{proof}

\begin{prop}\label{lemma make proper}Let $X$ be a smooth algebraic space and $\phi:X\to \cM$ an admissible period map.  Then there exists a log smooth partial compactification $X\subset \tilde X$ for which the period map extends to a proper map $\ol\phi: \tilde X\to\cM$. 
\end{prop}

\begin{proof}  Let $\bar X$ be a log smooth compactification of $X$.  For any polydisk $P$, consider the positive octant $\R_{\geq 0}^{r_P}$.  The assignment of the monodromy operator $N_i$ to the standard basis vector $e_i$ yields a linear map $\R^{r_P}\to\frak{g}$ to the Lie algebra. Its kernel is an integral linear subspace of $\R^{r_P}$, and we denote by $K$ its intersection with $\R^{r_P}_{\geq 0}$.  We may find an integral simplicial subdivision of the standard fan on $\R_{\geq 0}^{r_P}$ for which $K$ is a union of facets.  This subdivision corresponds to a (global) monomial modification ${\bar X}_P\to \bar X$ for which the condition of Lemma \ref{lemma proper cond} is satisfied on the preimage of $P$, once we extend the period map over the boundary components with no monodromy using the nilpotent orbit theorem \cite[Proposition 4.3]{bbkt}. Notice that any further monomial modification $\bar Z \to {\bar X}_P$ will also satisfy the condition above $P$.  Thus, taking $\bar Z$ to be a monomial modification of $\bar X$ that dominates each of the ${\bar X}_P$ and extending the period map over the boundary components with no monodromy, the condition of Lemma \ref{lemma proper cond} is satisfied. 
\end{proof}

\begin{eg}Unlike in the pure case, some blow-ups may be necessary.  Consider the mixed period space $\G_m=\Ext^1_{\Z MHS}(\Z(0),\Z(1))$ and the period map $\Abb^1\times\G_m\to\G_m$ which is just the second projection.  Take $\Abb^1\times\G_m\subset\Abb^2\subset \PP^2$ as a log smooth compactification.  On each vertical line $\Abb^1\times\{z\}$ the period map extends as it is trivial, but the period map does not globally extend over the line at infinity.  In this case if the monodromy logarithm around $\Abb^1\times\{0\}$ is $ N$, then the monodromy around infinity is $-N$. 
\end{eg}

\subsection{Algebraicity of the Hodge filtration}
For any definable analytic space $X$ and a $\C$-local system $\cE$, $\cE\otimes_{\C_X}\O_X$ is naturally a definable analytic coherent sheaf by taking flat trivialization on a definable cover by simply-connected open sets.  The following is a nonreduced version of the Deligne extension which is essentially contained in \cite[\S 5]{bbt}.  
\begin{prop}   Let $X$ be an algebraic space and $\cE$ a $\C$-local system whose local monodromy has unit norm eigenvalues.  Then the definable analytic coherent sheaf $E_X:=\cE\otimes_{\C_X}\O_{X}$ is algebraic.
\end{prop}
\begin{proof}  First assume $X^{\reduced}$ is smooth.  Let $Z$ be a compactification of $X$ for which $(Z^\reduced,Z^\reduced \backslash X^{\reduced})$ is log smooth.  We may take a $\R_\an$-definable cover of $Z^\an$ by open sets $P$ for which $P^\reduced\cong \Delta^n$ is a polydisk and $(P^*)^\reduced\cong(\Delta^*)^r\times\Delta^s$ where $P^*=P\cap X$.  As $P$ is Stein (since $\Delta^n$ is), we may lift the coordinates to functions $q_i$ on $P$, which are $\R_\an$-definable after shrinking $P$.  Now the $\R_{\an,\exp}$-definable analytic space structure on $P^*$ induces one on any chosen $\R_{\an,\exp}$-definable simply-connected fundamental set of the covering map $\HH^r\times \Delta^s\to(\Delta^*)^r\times \Delta^s$; call this space $\Phi$.  The $q_i$ are $\R_{\an}$-definable morphisms $P\to\C$ and as the multivalued function $\log:\C^*\to\C$ is definable on angular sectors, the logarithms $z_i=\log q_i$ are $\R_{\an,\exp}$-definable on $\Phi$.  We then define a Deligne extension $\bar E$ of $E_X$ locally using the lattice $\tilde v:=\exp(\sum_iz_iN_i)v$ for flat sections $v$ of $\cE$ where the $N_i$ are logarithms of the local monodromy, and the same proof as in the reduced case shows these extensions patch (see for instance \cite[Proposition 5.4]{Deligne_book}).  Now by ordinary GAGA \cite{gaga}, the extension $\bar E$ is algebraic, and an algebraic frame can be written analytically (hence $\R_\an$-definably) in terms of the $\tilde v$, while the change-of-basis to the flat frame $\exp(\sum_i z_iN_i)$ is $\R_{\an,\exp}$-definable as the $N_i$ are imaginary. 

In the general case, by performing blow-ups along reduced centers we may produce a proper map $\pi:Y\to X$ which is dominant on an open set $U$ of $X$ and for which $Y^\reduced$ is smooth.  Let $X'$ be the image of $Y$ in $X$.  For a sufficiently big thickening $S$ of the reduced complement $X^\reduced\backslash U^\reduced$, the following square is a pushout
\[\xymatrix{
S\times_X X'\ar[d]\ar[r]&X'\ar[d]\\
S\ar[r]&X.
}\] 
As $\pi':Y\to X'$ is dominant and $E_{X'}\subset \pi'_*E_Y$, by definable GAGA \cite[Theorem 3.1]{bbt} $E_{X'}$ is algebraic, while by Noetherian induction $E_{S}$ is algebraic.  $E_X$ is the pushout of $E_S$ and $E_{X'}$, hence algebraic.
\end{proof}

\begin{cor}\label{cor tfae}Let $X$ be an algebraic space with an analytic period map $\phi:X\to\cM$ whose reduction $\phi^{\reduced}:X^{\reduced}\to\cM$ is admissible.  Then the following are equivalent:
\begin{enumerate}
\item $\phi$ is definable (or equivalently admissible);
\item The Hodge filtration pieces $F^\bullet_X$ are definable analytic subbundles of the ambient flat vector bundle; 
\item The Hodge filtration pieces $F^\bullet_X$ are algebraic subbundles of the ambient flat vector bundle. 
\end{enumerate}
\end{cor}
\begin{proof}$(2)\Leftrightarrow(3)$ is immediate given the proposition and definable GAGA.  For $(1)\Leftrightarrow(2)$, let $U_i$ be a definable cover of $X$ by simply-connected open sets.  The definability of $\phi$ (given the definability of $\phi^{\reduced}$) is equivalent to the definability of the lifts $U_i\to M$ to the universal cover $M$ of $\cM$, which is in turn clearly equivalent to the definability of $F^\bullet_X$ as a subbundle of the ambient flat bundle with its flat definable structure.
\end{proof}

\begin{remark}  Recall by \cite{sz} that all variations of graded-polarized integral mixed Hodge structures coming from geometry are admissible.  From the corollary it is clear that this is true over possibly non-reduced bases as well.
\end{remark}

To algebraize theta bundles in Section \ref{sect theta}, we will need the following result, which formalizes the idea that the deformation theory of variations of Hodge structures is algebraic, even in the singular setting.

\begin{prop}\label{prop nbhd}  Let $\cM$ be a graded-polarized mixed period space and $X\subset\cM$ an algebraic Griffiths-transverse closed definable analytic subspace.  Then for any $n$ the $n$th order thickening of $X$ in $\cM$ is algebraic.  
\end{prop}
\begin{proof}By definable GAGA we may assume $X$ is reduced.  Let $(\mathcal{V}_\C,W_\bullet \mathcal{V}_\C)$ be the filtered $\C$-local system underlying the mixed variation on $X$, and let $(V,W_\bullet V,F^\bullet V)$ be the associated bifiltered vector bundle with its canonical algebraic structure.  Consider $\mathrm{Fl}=\mathrm{Fl}(W_\bullet V)$ the relative flag variety of filtrations $F'^\bullet V$ of $V$ which intersect $W_\bullet V$ with the same dimensions as $F^\bullet V$ and for which the induced filtration $F'^\bullet\gr_k^WV$ is $q_k$-isotropic for each $k$.  We have a section $s:X\to \mathrm{Fl}$ of the natural map $\pi:\mathrm{Fl}\to X$ given by $F^\bullet V$; let $S_n\subset \mathrm{Fl}$ be the $n$th order thickening of $s(X)$ in $\mathrm{Fl}$, which is clearly algebraic.

There is a natural admissible period map $\iota:S_n\to \cM$ extending the inclusion $X\subset\cM$ which we claim is the closed embedding of the $n$th order thickening.  Indeed, $S_n$ can also be analytically constructed as follows.  Let $\tilde X$ be the universal cover of $X$, so we have a closed embedding $\tilde X\subset M$ where $M$ is the universal cover of $\cM$.  Let $\tilde S_n$ be the $n$th order thickening of $X$ in $M$, considered as a subspace of $M\times\check{M}$ via the diagonal embedding $M\to M\times \check{M}$ where $\check M$ is the ``compact" dual.  Now if $\Gamma$ is the image of the monodromy representation, then we have an embedding $S^\an_n=\Gamma\backslash \tilde S_n\subset \mathrm{Fl}=\Gamma\backslash(M\times\check{M})$ where $\Gamma$ acts diagonally on $M\times \check{M}$.  

That $S_n$ is analytically the closed embedding of the $n$th order thickening is now obvious, and the definability of $\iota:S_n\to \cM$ follows from Corollary \ref{cor tfae}.

\end{proof}

\subsection{Algebracity of images}


%
%
%

\begin{prop}\label{prop images}Let $X$ be an algebraic space and $\phi:X\to \cM$ a definable mixed period map.  Then there is a factorization
\[\begin{tikzcd}
X\arrow{rr}{\phi}\arrow{dr}[swap]{f}&&\cM.\\
&Y\arrow[ru,hook,swap,"\iota"]&
\end{tikzcd}\]
where $f$ is dominant algebraic and $\iota$ is a closed immersion.  
\end{prop}
\begin{proof} First assume $X$ reduced and let $\pi: X'\to X$ be a resolution.  By Lemma \ref{lemma make proper} there is a partial compactification $X'\subset Z'$ for which the period map of $X'$ extends to a proper map $\bar\phi: Z'\to\cM$.  Now apply \cite[Theorem 4.2]{bbt}. 

In general, let $Y$ be the closure of the image of $X^\reduced$, which is algebraic by the above.  By definable GAGA, the pullback $X_n\subset X$ of the $n$th order thickening of $Y$ to $X$ is an increasing sequence of subspaces set-theoretically supported on all of $X^\reduced$, which by Noetherian induction \cite[Cor. 2.32]{bbt} on the supports of the ideal sheaves $I_{X_n}$ is eventually all of $X$.  By Proposition \ref{prop nbhd} and definable GAGA, the claim for $X$ follows.
\end{proof}

\begin{cor}\label{cor mon image}  Let $\Gamma_\mathrm{mon}\subset\G(\Z)$ be the image of the monodromy representation of the variation of mixed Hodge structures associated to $\phi$, and $\phi_\mathrm{mon}:X\to \Gamma_\mathrm{mon}\backslash M$ the corresponding lift of $\phi$.  Then there is a factorization
\[\begin{tikzcd}
X\arrow{rr}{\phi_\mathrm{mon}}\arrow{dr}[swap]{f}&&\Gamma_\mathrm{mon}\backslash M.\\
&Y\arrow[ru,hook,swap,"\iota"]&
\end{tikzcd}\]
where $f$ is dominant algebraic and $\iota$ is a closed immersion.  
\end{cor}
\begin{proof}  As in the above proof we may assume $\phi$ and therefore $\phi_\mathrm{mon}$ is proper.  Taking $\Gamma'_\mathrm{mon}\subset\Gamma_\mathrm{mon}$ to be a finite-index torsion-free normal subgroup, $Y$ will be the quotient of the image of $X$ in $\Gamma_\mathrm{mon}'\backslash M$ by $\Gamma_\mathrm{mon}/\Gamma_\mathrm{mon}'$, so we may assume $\Gamma_\mathrm{mon}$ to be contained in a torsion-free normal arithmetic subgroup $\Gamma\subset\G(\Z)$.  Now $Y$ is a finite \'etale cover of the image in $\Gamma\backslash M$ and therefore algebraic.
\end{proof}
\subsection{Theta bundles}\label{sect theta}
Let $\cD$ be a polarized pure period space parametrizing polarized weight $-1$ Hodge structures $V$ on $(V_\Z,q_\Z)$.  We can consider the graded-polarized mixed period spaces $\cM$ resp. $\cM'$ of extensions   
\[0\to V\to E\to \Z(0)\to 0\]
resp. 
\[0\to \Z(-1)\to E\to V\to 0\]
both of which map to $\cD$.  We may also consider the graded-polarized mixed period space $\cB$ parametrizing mixed Hodge structures $E$ with weights $[-2,0]$ with 
\begin{align*}
\gr^W_{-2}E&=\Z(-1)\\
\gr^W_{-1}E&=V\\
\gr^W_{0}E&=\Z(0).
\end{align*}
The natural map $\cB\to\cM\times_\cD\cM'$ is canonically an analytic $\Ext^1_{\Z MHS}(\Z(0),\Z(-1))\cong\G_m$ torsor which we call the biextension torsor; the associated analytic line bundle $\cP$ on $\cM\times_\cD\cM'$ we call the biextension bundle.  Viewing $\cM\to\cD$ as the universal intermediate Jacobian
 \[J(\cV):=\cV_\C/F^0\cV+\cV_\Z\]
 and $\cM'\to\cD$ as the dual $J(\cH^\vee)$, the biextension bundle $\cP$ is naturally thought of as the universal Poincar\'e bundle.  
 \begin{remark}
 While the total space $\cB$, the map to $\cM\times_\cD\cM'$, and the $\G_m$-action are all definable analytic, it is not clear that $\cB$ is a definable analytic $\G_m$-torsor as it is not clearly that it is definably locally trivial.
 \end{remark}

\begin{prop}\label{prop theta alg gen} Let $X$ be an algebraic space and $\phi: X\to \cM\times_\cD\cM'$ an admissible period map.  Then the pullback $B_X$ of the biextension torsor has a natural algebraic \'etale $\G_m$-torsor structure for which the map $B_X\to\cB$ is definable.
\end{prop}
\begin{proof}We start by making some preliminary observations.  First, note that $B_X$ has a natural definable structure, as it is the base-change of $\cB$.  Thus, it suffices to show that the space underlying $B_X$ has an algebraic structure compatible with the definable structure.  Indeed, this algebraic structure is unique by definable GAGA, hence the naturality.  Since both the map to $X$ and the $\G_m$-action are pulled back from $\cB$, they are likewise algebraic, and as $B_X\to X$ clearly admits an fppf-local section (over $B_X$ for instance) it follows that it is an \'etale $\G_m$-torsor. 

Now to show that the underlying space of $B_X$ is algebraic we proceed by considering successively more general cases.
\subsubsection*{Step 1}  For $X$ smooth, the proposition is a result of Brosnan--Pearlstein:
\begin{thm}[Thm. 241 of \cite{bparch}]\label{thm theta alg smooth case}  In the above situation and assuming $X$ smooth, $B_X$ (as a sheaf) admits a natural meromorphic extension to any log smooth compactification $\bar X$ whose sections correspond to admissible liftings $\tilde\phi:X\to\cB$ of $\phi$.  In particular, $B_X$ is an \'etale $\G_m$-torsor.
\end{thm}

Note that this algebraic structure is indeed compatible with the definable structure:  \'etale locally the map $B_X\to \cB$ is identified with 
\[B_{X}\cong \G_m\times X\to\G_m\times \cB\to\cB\]
where the left isomorphism is induced by a local section of $B_{X}\to X$, the middle map comes from the corresponding admissible lift $X\to \cB$, and the right map is the action.  

\subsubsection*{Step 2} For $X$ reduced, by taking the closure of the image we may assume $\phi:X\to \cM\times_\cD\cM'$ is a closed immersion by Proposition \ref{prop images}.  Let $\pi:Y\to X$ be a resolution.  By the previous step, $B_{Y}$ is algebraic and the natural map $B_{Y}\to\cB$ is an admissible period map.  It follows by Proposition \ref{prop images} again that the image of $B_{Y}\to\cB$ is algebraic, and this is just the underlying space of $B_X$.  

\subsubsection*{Step 3} For general $X$ we may still assume $\phi:X\to \cM\times_\cD\cM'$ a closed immersion.  By the previous step $B_{X^\reduced}$ is algebraic, and by Proposition \ref{prop nbhd} and definable GAGA we conclude that the total space $B_X$ is algebraic, hence $B_X$ is algebraic.

\end{proof}

\begin{cor}  In the setup of the proposition, the pullback $P_X$ of the biextension bundle is naturally an algebraic line bundle.
\end{cor}
Note that we have a natural definable map
\[\sigma:J(\cV)\to J(\cV^\vee)\]
commuting with the projection to $\cD$ which on fibers is the map $\Ext^1_{\Z MHS}(\Z(0),V)\to\Ext^1_{\Z MHS}(\Z(0),V^\vee)$ coming from the polarizing form $q:V\to V^\vee$.

\begin{defn}  Let $X$ be an algebraic space and $\phi:X\to\cM$ an admissible period map.  The line bundle $\Theta_X$ on $X$ which is the pullback of $\cP$ along $\phi\times(\sigma\circ\phi):X\to\cM\times_\cD\cM'$ endowed with the natural algebraic structure of Proposition \ref{prop theta alg gen} is the \emph{theta bundle} of $\phi$.  
\end{defn}

\begin{prop}  Let $X$ be an algebraic space, $V$ a polarized pure Hodge structure of weight $-1$, and $\phi:X\to J(V)$ an admissible quasifinite period map.  Then the image $\phi(X)$ is contained in a translate of a subtorus which is an abelian variety.
\end{prop}

\begin{proof}  By Proposition \ref{prop images} we may assume $X$ is proper, reduced, irreducible, and a closed subspace $X\subset J(V)$ containing the split point $0\in J(V)$.  By replacing $X$ with the image of the difference map $X\times X\to J(V)$ we may eventually assume $X$ is a sub-group.  If $H_\Z$ is the image of the monodromy $H_1(X,\Z)\to V_\Z$ and $H_\C\subset V_\C$ the complex span, then $X=H_\C/F^0V\cap H_\C+H_\Z$.  The tangent bundle of $J(V)$ is canonically $V_\C/F^0V$ and the Griffiths transverse subbundle is $F^{-1}V/F^0V$, so we have $H_\Z\subset F^{-1}V$.  As $X$ is definable, it must be a compact real torus, so we must have $H_\R\cong H_\C/F^0V\cap H_\C$ via the quotient map.  It follows that $H_\Z$ underlies a polarized sub Hodge structure of level one.  
\end{proof}

\begin{cor}\label{cor theta ample one}  In the setup of the proposition, $\Theta_X$ is ample.
\end{cor}

\begin{proof}
The theta bundle on $J(V)$ is clearly the line bundle associated to the hermitian form $q(u,\ol v)$, and restricts to the usual theta bundle of $H_\C/F^0H+H_\Z$.
\end{proof}

\begin{remark}\label{rmk new proof}Considerations as in the previous proposition can be used to give a new proof of Theorem \ref{thm theta alg smooth case} as follows.  Consider a diagram
\[\xymatrix{
X\ar[r]\ar[d]&\cM\times_\cD\cM'\ar[d]\\
Y\ar[r]&\cD
}\]
where the horizontal maps are Griffiths transverse closed immersions and $X,Y$ are reduced.  After base-changing along an \'etale map $Y'\to Y$ with dense image, $X':=X\times_YY'\to Y'$ admits a section.  As in the proof of the proposition, using Proposition \ref{prop images} we may replace $X'$ with the image of the difference map $X'\times_{Y'}X'\to \cM\times_\cD \cM'$, and after finitely many iterations we may assume (after shrinking $Y'$) there is a factorization 
\[\xymatrix{
X'\ar[rd]\ar[r]&A\times_{Y'}A^\vee\ar[r]\ar[d]&\cM\times_\cD\cM'\ar[d]\\
&Y'\ar[r]&\cD
}\]
where $A\to Y'$ is a smooth definable analytic family of polarized abelian varieties whose fibers are the subgroups generated by the corresponding fibers of $X'\to Y'$.  Then $A$ is pulled back along a definable hence algebraic map of $Y'$ to a Shimura variety and the universal Poincar\'e bundle is algebraic, so $A$, $B_{A\times_{Y'}A^\vee}$, and therefore $B_{X'}$ are all algebraic.  The closure of the image of $B_{X'}$ in $\cB$ is $B_X$, and therefore algebraic by Proposition \ref{prop images}.

\end{remark}

While the argument in the remark directly shows that $B_X$ is algebraic on a stratification, the global definable structure is needed to glue these algebraic structures together since the fiber dimension may jump, as shown in the example(s) below.

It is not hard to provide examples where fiber dimensions jump, even in the mixed Shimura setting. In fact, consider the universal abelian variety $\mathcal{X}_g$ over the moduli stack 
$\A_g$ (one may add level structure to rigidify everything into schemes). Now one may simply take a curve $C$ inside a fiber over $\A_g$, and a generic surface $S$ containing it.
This is an example of a mixed variation of weights 0,1 with jumping fiber dimensions over the associated graded.

We give below what we consider a more interesting example, where the fibers over the graded are generically finite simply for lack of hodge classes in the associated pure variation, but then over points in the graded which acquire hodge classes, the fiber dimension jumps. This kind of example is harder to construct ``artificially'' in the manner above, and appears to be
a more intrinsic geometric phenomenon.

\begin{eg}  Let $K$ be a sufficiently high level cover of the moduli space of K3 surfaces polarized by the lattice $\mat{0}{2}{2}{0}$, so that it is an irreducible quasiprojective variety.  Let $f_1,f_2$ be the divisor classes of the two elliptic pencils and let $S$ be the moduli space of pairs $(X,E)$ with $X\in K$ and $E$ a smooth section of $f_1$.  $S$ is also irreducible and admits a forgetful map $S\to K$.  Consider the cohomology of the complement $H^2(X\backslash E,\Z)$ which sits in an extension
\[0\to H^2(X,\Z)/\Z f_1\to H^2(X\backslash E,\Z)\to H^1(E,\Z)(-1)\to 0.\]
Let $H^2_v(X,\Z):=H^2(X,\Z)/(\Z f_1+ \Z f_2)$, which yields an extension 
 \[0\to H^2_v(X,\Z)\to H^2_v(X\backslash E,\Z)\to H^1(E,\Z)(-1)\to 0.\]

Let $\phi:S\to \cM$ be the resulting mixed period map for the variation $H^2_v(X\backslash E,\Z)$, and $\gr\phi: S\to \cD$ that of the associated graded.  It is easy to see that $H^2_v(X,\Z)\otimes H^1(E,\Z)^\vee(1)$ generically has no nontrivial sub Hodge structures, so the generic fiber of $\gr\phi$ is 0-dimensional.  On the other hand, if $X$ is the Kummer surface of $E\times E'$ with the elliptic pencil given by the first factor $E\times\{0\}$, then for $p\in E\backslash E[2]$ we have 
\begin{align*}
H^2(X\backslash E\times\{p\},\Q)&=H^2(\mathrm{Bl}_{E[2]\times E'[2]}(E\times E'\backslash\{\pm p\}),\Q)^{\pm 1}\\
&=\Q(-1)^{17}\oplus H^1(E,\Q)\otimes H^1(E'\backslash\{\pm p\},\Q) 
\end{align*}
so the associated graded is constant on the fiber of $S\to K$ above $X$.
\end{eg}

\section{Setup for the proof of quasiprojectivity}\label{sect setup}

In this section we collect several results that will be needed for the proof of the quasiprojectivity part of Theorem \ref{thm main}.  The first is an ampleness criterion in terms of definable sections; the second allows us to endow the cohomology groups of variations of mixed Hodge structures with mixed Hodge structures over arbitrary bases; the third gives some control on the monodromy of extensions of variations of mixed Hodge structures, again over arbitrary bases.

\subsection{Definable-analytically quasiaffine maps}
We first show that to prove $f$-ampleness of an algebraic line bundle $L$ and an algebraic map $f:X\to Y$, it suffices to show the definable stalks of $f_*L^n$ separate points.  Note that this is weaker than the assumption that $f_*L^n$ separates points definably locally on $Y$.
\begin{prop}\label{lemma definably ample}Let $f:X\to Y$ be a map of reduced algebraic spaces, $L$ a line bundle on $X$.  Assume for any point $y\in Y$ and any 0-dimensional subspace $P\subset X$ supported on the fiber $X_y$ above $y$ that the restriction on stalks
\[(f^\df_*(L^n)^\df)_y\to (f^\df_*(L^n|_P)^\df)_y\]
is surjective for $n\gg 1$.  Then $L$ is $f$-ample.
\end{prop}
\begin{proof}By Zariski's main theorem, it is enough to show for all $y$ and $P$ as in the statement of the theorem that the restriction map
\[(f_*(L^n))_y\to (f_*(L^n|_P))_y\]
is surjective for $n\gg1$.  Let $g:Z\to Y$ be a relative compactification of $X$, so $g$ is proper and there is an open immersion $X\to Z$ over $Y$.  Let $S$ be the complement of $X$ in $Z$.  By assumption there is an $n$, an analytic open neighborhood $U\subset Y^\an$ of $y$ and finitely many sections of $(L^n)^\df(f^{-1}(U))$ separating $P$, since $(f^\df_*(L^n|_P)^\df)_y$ is a finite-dimensional vector space.  We may assume there is a line bundle $M$ extending $L^n$, and by the following lemma definable sections extend meromorphically.
\begin{lem}  Let $Z$ be a reduced definable analytic space and $S\subset Z$ a closed definable analytic subspace.  Any definable analytic $f:Z\backslash S\to \C$ extends meromorphically to $Z$. 
\end{lem}
\begin{proof}The closure of the graph $\Gamma(f)\subset (Z\backslash S)\times\C$ in $Z\times\PP^1$ is definable and analytic by for example Bishop's theorem \cite[Theorem 3]{bishop}, as definable sets have locally bounded volume.
\end{proof}
It thus follows that
\[(g^\an_*(\sHom(I_S^m,M)^\an)_y\to (g^\an_*(L_P^n)^\an)_y\]     
is surjective for $m\gg0$, and by ordinary GAGA this means the horizontal map below is surjective, finishing the proof.
\[\xymatrix{
(g_*(\sHom(I_S^m,M))_y\ar[rd]\ar[rr]&& (g_*L_P^n)_y\\
&(f_*L^n)_y\ar[ur]&
}\]  
\end{proof}
Lemma \ref{lemma definably ample} provides a particularly easy criterion for $X\to Y$ to be quasiaffine. 
\begin{defn}We say that a map $X\to Y$ of definable analytic spaces is definable-analytically quasiaffine if \emph{analytically} locally on $Y$ it factors as
\[\xymatrix{
X\ar[rd]\ar[r]^\iota&\C^N\times Y\ar[d]^{\pi_2}\\
&Y
}\]
where $\iota$ is a definable analytic locally closed immersion and $\pi_2$ the second projection.
\end{defn}
Recalling that an algebraic map $X\to Y$ is quasiaffine if and only if $\O_X$ is relatively ample \cite[II.5.1.6]{egaii}, we have:
\begin{cor}\label{cor qa crit}Let $f:X\to Y$ be a map of algebraic spaces which is definable-analytically quasiaffine.  Then $f$ is quasiaffine.
\end{cor}

\subsection{Hodge modules and period maps}  To equip the cohomology of variations of mixed Hodge structures over arbitrary bases with functorial mixed Hodge structures, we will rely crucially on Saito's formalism of mixed Hodge modules \cite{saito88,saito90}.  Briefly, for any reduced algebraic space $X$ there is an abelian category $\MHM(X)$ of graded polarizable mixed Hodge modules and a faithful functor
\[\rat:D^b\MHM(X)\to D^b_c(\Q_X)\] 
which is exact with respect to the perverse $t$-structure and such that the usual functors $Rf_*, f^*, f_!, f^!,\otimes^L,R\sHom$ on derived categories of constructible sheaves lift to functors $f_*, f^*, f_!, f^!,\otimes,\sHom$.  For $X$ smooth, a mixed Hodge module consists of a filtered $D$-module $M$ and a $\Q$-perverse sheaf $P$ with a quasi-isomorphism $\mathrm{DR}(M)\xrightarrow{\cong}P_\C$ where $\mathrm{DR}(M)$ is the de Rham complex of $M$, while in the general case they are patched together from such objects via local embeddings into smooth ambient spaces.  

\begin{defn}For $X$ a reduced algebraic space we say $E\in D^b\MHM(X)$ is \emph{smooth} if its underlying rational structure is a local system in degree 0 (with respect to the standard $t$-structure on $D^b_c(\Q_X)$). 
\end{defn}

For smooth $X$, there is a natural equivalence of categories \cite[Theorem 3.27]{saito90}
\begin{equation}\left\{\parbox{5.5cm}{\centering admissible variations of rational mixed Hodge structures on $X$}\right\}\label{eq equiv}\to\left\{\mbox{smooth objects of $D^b\MHM(X)$}\right\}\end{equation}
which is compatible with pull-backs along algebraic maps $f:X\to Y$.  

\begin{prop}\label{prop mhm}  The functor \eqref{eq equiv} uniquely extends to a fully faithful functor for any reduced algebraic space $X$ which is compatible with pull-backs along algebraic maps $f:X\to Y$.  If $X$ is moreover seminormal then the extension is an equivalence of categories.
\end{prop}

In particular, to every admissible period map $X\to\cM$ we obtain a ``pullback object"\footnote{$E^H_X$ of course depends on $\phi$.} $E^H_X\in D^b\MHM(X)$ whose underlying rational structure is the local system $\cE_X$.

\begin{proof}The uniqueness and functoriality are consequences of the uniqueness and functoriality for smooth $X$ and the following fact: 
\begin{lemma}\label{lemma mhm morph} Let $X$ be a reduced algebraic space. For smooth $E,F\in D^b\MHM(X)$ and any dense open set $j:U\to X$ we have 
\[\Hom(E,F)\cong \Hom(j^*E,j^*F)\]
  via the natural map.  
  \end{lemma}
\begin{proof}  Let $\cE=\rat(E)$ and $\cF=\rat(F)$.  On the level of sheaves
  \[\Hom(\cE,\cF)\cong \Hom(j^*\cE,j^*\cF)\]
via the natural map, while
\begin{align}
\Hom(E,F)&=\Hom(\Q_X^H,\sHom(E,F))\notag\\
&=\Hom(\Q(0),\pt_*\sHom(E,F))\label{eq saito}\\
&=\Hdg_0(\Hom(\cE,\cF))_\Q\notag
\end{align}
where we equip $\Hom(\cE,\cF)$ with its mixed Hodge structure as $\rat(H^0\pt_*\sHom(E,F))$ and define $\Hdg_k(H)_\Q:=\Hom(\Q(-k),H)=F^kH\cap W_{2k}H_\Q$ in general for a rational mixed Hodge structure $H$.  Likewise for $\Hom(j^*E,j^*F)$. 
\end{proof}
It therefore suffices to show the existence of an extension of $E^H_{X^{\mathrm{reg}}}$ on the regular locus $X^{\mathrm{reg}}\subset X$ to an object $E\in D^b\MHM(X)$ with rational structure $\cE_X$.  We proceed by induction on $\dim X$, the 0-dimensional case being obvious.  

Let $\pi:X'\to X$ be a resolution, let $S\subset X$ be the singular locus, and let $S'\subset X'$ be the reduced preimage of $S$.  By induction we may assume $E^H_S$ and $E^H_{S'}\cong \pi^*E_S$ exist, and by stipulation $E^H_{X'}$ exists.  We have a triangle in $D^b\MHM(X)$
\begin{equation}\label{eq q cone}\Q_X^H\to \pi_*\Q_{X'}^H\to A\to\Q_X^H[1].\end{equation}
As $\pi$ is proper the middle map has an adjoint
\begin{equation}\notag\label{eq q map}\Q_{X'}^H\to \pi^!A.\end{equation}
Consider the map
\[E^H_{X'}\to E^H_{X'}\otimes \pi^!A\]
obtained from tensoring by $E^H_{X'}$.  We have natural identifications
\[E^H_{X'}\otimes \pi^!A\cong E^H_{S'}\otimes \pi^!A\cong \pi^!(E^H_S\otimes A)\] 
the first because $\pi^!A$ is supported on $S'$ and the second because $E^H_{S}$ is smooth.  We therefore obtain a map $E^H_{X'}\to \pi^!(E^H_S\otimes A)$, and we define $E$ to be the cone of the adjoint:
\begin{equation}\label{eq E cone}E\to \pi_*E^H_{X'}\to E^H_S\otimes A\to E[1].\end{equation}
The image of \eqref{eq E cone} under $\rat$ is easily seen to be isomorphic to the natural sequence
 \[\cE_X\to R\pi_*\cE_{X'}\to \cE_S\otimes \rat(A)\to\cE_X[1].\]
 obtained by tensoring the $\rat$ of \eqref{eq q cone} by $\cE_X$.  Moreover, restricting \eqref{eq E cone} to the regular locus we see (by proper base-change) that $E^H_{X^\mathrm{reg}}\cong E^H_{X'}|_{X^\mathrm{reg}}\cong E|_{X^\mathrm{reg}}$.
 
For the second claim, assume that $X$ is seminormal and let $\pi:X'\to X$ be a resolution. Let $E\in D^b\MHM(X)$ be a smooth object and let us prove that it comes from an admissible variation of rational mixed Hodge structures on $X$. Since we can argue Zariski-locally, let's assume that $\pi^*E$ is associated to a period map $\phi:X'\to\cM$. Clearly $\phi$ is pointwise constant on any fiber of $\pi$.  Since $\cM$ is smooth and $X$ seminormal, it follows that $\phi$ factors through $X$.  Here we've used that the analytification of a seminormal algebraic space is weakly normal \cite[Cor. 6.14]{gtsn} so that the regular functions are continuous meromorphic functions. \end{proof}

\begin{remark} \label{remark should}
The seminormality hypothesis is necessary in the second statement of Proposition \ref{prop mhm} as in general the seminormalization $X'\to X$ is a universal homeomorphism and the functor $\rat :D^b\MHM(X)\to D_c^b(\Q_X)$ is faithful.

\end{remark}

\subsection{Monodromy of extensions}
Let $Y$ be a reduced algebraic space with an admissible variation of rational mixed Hodge structures $V_Y$.  Let $X$ be a reduced algebraic space with a map $f:X\to Y$ and an admissible variation of rational mixed Hodge structures $E_X$ which sits in an extension
\[0\to V_X\to E_X\to \Q_X(0)\to0\]
where $V_X=f^*V_Y$.  We have a corresponding exact sequence of rational structures
\[0\to \cV_X\to\cE_X\to\Q_X\to0.\]
Given a point $y\in Y$, let $U\subset Y$ be a small neighborhood.  The local system $\cE_X$ restricted to $X_U:=f^{-1}(U)$ has monodromy landing in $\cV_{Y,y}$, and we will need two lemmas controlling the image.
\begin{lemma}\label{lemma saito}In the above situation, the image of the extension class of $\cE_X$ in $\Ext^1(\Q_X,\cV_X)\cong H^1(X,\cV_X)$ under the composition
\begin{equation}\label{eq sequence}H^1(X,\cV_X)\to H^1(Y,Rf_*\cV_X)\to (R^1f_*\cV_X)_y\end{equation}
is Hodge of weight 0 for any $y\in Y$.
\end{lemma}
Note that in the statement of the lemma we are using that $\cV_X$ underlies an object $V^H_X\in D^b\MHM(X)$ by Proposition \ref{prop mhm}, and that the sequence \eqref{eq sequence} underlies
\[H^1\pt_*(V^H_X)\to H^1\pt_*(f_*V^H_X)\to H^1(i^*_yf_*V^H_X)\]  
where $i_y:y\to Y$ is the inclusion.  All three groups in \eqref{eq sequence} are therefore equipped with mixed Hodge structures and the maps are morphisms of mixed Hodge structures.
\begin{proof}    The triangle
\[\cV_X\to\cE_X\to\Q_X\to\cV_X[1]\]
lifts to a triangle 
\[V^H_X\to E^H_X\to \Q^H_X\to V_X^H[1]\]
in $D^b\MHM(X)$ as the morphisms exist by Lemma \ref{lemma mhm morph} and exactness can be checked on the underlying rational structures.  Moreover,
\[\Hom(\Q(0),H^1\pt_*V_X)=\Hom(\Q(0),\pt_*V^H_X[1])=H^1\pt_*(V^H_X)\]
and the group on the left is the weight 0 Hodge classes of $H^1(X,\cV_X)$.
\end{proof}
Note that $f_*V^H_X=V^H_Y\otimes f_*(\Q^H_X)$ and that $$H^1(i_y^*f_*V^H_X)\cong i_y^*V^H_Y\otimes H^1(i_y^*f_*\Q^H_X)\cong \Hom(H_1(X_U,\Q),\cV_{Y,y})$$
as mixed Hodge structures.  Furthermore, under this identification the image of the extension class of $\cE_X$ under \eqref{eq sequence} is precisely the monodromy representation of $\cE_X$ restricted to $X_U$.  

The previous lemma therefore implies that the monodromy representation $H_1(X_U,\Q)\to \cV_{Y,y}$ is a morphism of mixed Hodge structures; the following lemma controls the Hodge numbers of $H_1(X_U,\Q)$.
\begin{lemma}\label{lemma numbers}  For any map $f:X\to Y$ of reduced algebraic spaces and any $y\in Y$, the nonzero Hodge numbers $h^{p,q}$ of $(R^nf_*\Q_X)_y$ satisfy $0\leq p,q\leq n$.
\end{lemma}
\begin{remark}\label{rmk Y pt}The claim is true for $Y$ a point (for instance see \cite[Thm. 5.39]{PSmix}), and therefore for proper $f$ by proper base-change.
\end{remark}
In the following proof the sheaves/morphisms between them naturally underlie objects in the derived category of mixed Hodge modules (and thus possess/preserve natural mixed Hodge structures), but we phrase the argument entirely in terms of the rational structures for simplicity.
\begin{proof}  We proceed by induction on $\dim X$, the claim being obvious if $X$ is 0-dimensional.  Choose a relative compactification $Z$ and let $F$ be the fiber over $y$.  
\[\xymatrix{
X\ar[dr]_{f}\ar[r]^j&Z\ar[d]^{g}&\ar[l]_i F\ar[d]\\
&Y&\ar[l]^{i_y}y
}\]
Let $Z'$ be a log resolution of $(Z,Z\setminus X)$ and let $X',F'$ be the reduced preimages of $X,F$.  \[\xymatrix{
S'\ar[r]^{\iota'}\ar[d]^\gamma&X'\ar[d]^\phi\ar[r]^{j'}&Z'\ar[d]^{\pi}&\ar[l]_{i'} F'\ar[d]^p\\
S\ar[r]_\iota&X\ar[r]_j&Z&\ar[l]^i F\\
}\]
The map $\phi$ is an isomorphism on a dense Zariski open set $V\subset X$; let $\iota:S\to X$ be the inclusion of the complement of $V$ and $\iota':S'\to X'$ the preimage.

There is a natural morphism $\Q_X\to R\phi_*\Q_{X'}$.  Let $A$ be the cone, so we have a triangle
\begin{equation}\Q_X\to R\phi_*\Q_{X'}\to A\to\Q_X[1]\label{eq cone}\end{equation}
and note that $\sH^i(A)=0$ for $i<0$.  Applying $R\Gamma i^*Rj_*$ we obtain an exact sequence
\begin{equation}\label{eq les}\xymatrix{
H^{n-1}(i^*Rj_*A)\ar[r]& H^n(i^*Rj_*\Q_X)\ar[r]& H^n(i^*Rj_*R\phi_*\Q_{X'})&\\
&(R^nf_*\Q_X)_y\ar@{=}[u]&H^n(i'^*Rj'_*\Q_{X'})\ar@{=}[u]&
}\end{equation}
 where the vertical identifications are by proper base-change.    The object $A$ is supported on $S$, so pulling \eqref{eq cone} back to $S$ we obtain a triangle 
\[\Q_S\to R\gamma_*\Q_{S'}\to \iota^*A\to\Q_S[1]\]
and after applying $R\Gamma i^*Rj_*\iota_*$ an exact sequence
\[\xymatrix{
H^{n-1}(i^*Rj_*\iota_*R\gamma_*\Q_{S''})\ar@{=}[d]\ar[r]& H^{n-1}(i^*Rj_*A)\ar[r]& H^n(i^*Rj_*\iota_*\Q_S)\ar@{=}[d]\\
(R^{n-1}(f\circ\iota\circ\gamma)_*\Q_{S''})_y&& (R^n(f\circ \iota)_*\Q_S)_y.}\]
By the induction hypothesis, it follows that the nonzero Hodge numbers $h^{p,q}$ of $H^{n-1}(i^*Rj_*A)$ have $0\leq p,q\leq n$, and by \eqref{eq les} it is therefore enough to prove:
\begin{claim} The nonzero Hodge numbers $h^{p,q}$ of $H^n(i'^*Rj'_*\Q_{X'})$ satisfy $0\leq p,q\leq n$.
\end{claim}
\begin{proof}Recall that $Rj'_*$ is exact in the perverse $t$-structure; let $M=j'_*\Q_{X'}^H[d]$ where $d=\dim X$ (we may assume $X'$ irreducible).  Recall that
\[\gr^W_{d+k}M = \bigoplus_{|I|=k}\Q^H_{D_I}(-k)[d-k]\]
where $D_I=\bigcap_{i\in I}D_i$ is the $I$th boundary stratum of $Z'\backslash X'$.  Defining $F'_I:=F'\cap D_I$ to be the reduced intersection, we have $i'^*\Q^H_{D_I}=\Q^H_{F'_I}$ (as this is clearly true on the level of underlying rational structures), so
\[
i'^*\gr^W_{d+k}M=\bigoplus_{|I|=k}\Q^H_{F'_I}(-k)[d-k].
\]
To prove the claim, it suffices to show the claimed vanishing of Hodge numbers for $H^{n-d}(i'^*\gr^W_{d+k}M)$ for all $k$.  This in turn follows because $H^{n-k}(F'_I,\Q)(-k)$ satisfies the vanishing for all $k$ (see Remark \ref{rmk Y pt}).

\end{proof}

\end{proof}

\section{Proof of quasiprojectivity}\label{sect proof}

\subsection{Induction step}

Let $\cM_{w}$ be a mixed period space parametrizing graded-polarized integral mixed Hodge structures $E$ with weights $\leq w$ and let $\mathcal{M}_{ w-1}$ (resp. $\cM_{[w-1,w]}$, $\cD_{w-1}$) be the graded-polarized mixed period space parametrizing mixed Hodge structures $H$ of the form $W_{w-1}E$ (resp. $\gr_{[w-1,w]}^WE$, $\gr_{w-1}^WE$).  We naturally have a definable analytic morphism
\[\cM_{ w}\to\cM_{w-1}\times_{\cD_{w-1}} \cM_{[w-1,w]}\]
From section \ref{sect theta} we have an analytic theta bundle $\Theta_{[w-1,w]}$ on $\cM_{[w-1,w]}$; we also denote by $\Theta_{[w-1,w]}$ the pullback to $\cM_{w}$.  As no confusion is likely to occur, for any period map $\phi:X\to \cM_{w}$ we also denote by $\Theta_{[w-1,w]}$ the pullback to $X$ together with its natural algebraic structure as in Proposition \ref{prop theta alg gen}.

The main result of this section is:
\begin{prop}\label{prop induct step}Let $X,Y,Z$ be algebraic spaces together with a diagram
\[\xymatrix{
X\ar[dd]_{f}\ar[rd]^g\ar[rr]&&\cM_w\ar[dd]\ar[rd]&\\
&Y\ar[rr]\ar[ld]^h&&\cM_{w-1}\times_{\cD_{w-1}}\cM_{[w-1,w]}\ar[ld]\\
Z\ar[rr]&&\cM_{w-1}&
}\]
whose horizontal maps are definable Griffiths transverse closed immersions.  Then $\O_X$ is $g$-ample and the theta bundle $\Theta_{[w-1,w]}$ is $h$-ample.  
\end{prop}

\begin{cor}\label{cor theta ample}In the above situation, $\Theta_{[w-1,w]}$ is $f$-ample.
\end{cor}

Before the proof we make some observations.  As in the introduction, $\cM_{w}$ embeds into the mixed period space $\cM'_{ 0}$ parametrizing extensions of the form  
\[0\to W_{w-1}E\otimes \gr_w^WE^\vee \to E'\to\Z(0)\to0\]  
and likewise we have embeddings of $\cM_{w-1}, \cM_{[w-1,w]}, \cD_{w-1},\cM_{w-1}\times_{\cD_{w-1}}\cM_{[w-1,w]}$ into the corresponding spaces $\cM_{-1}', \cM_{[-1,0]}', \cD_{-1}',\cM'_{-1}\times_{\cD'_{-1}}\cM'_{[-1,0]}$ associated to $\cM'_{0}$ that are compatible with the obvious maps. Finally, $\Theta_{[w-1,w]}$ is pulled back from $\cM'_{0}$.  

Thus we may assume $w=0$ and $\gr_0^WE\cong\Z(0)$, and $\cM_0$ parametrizes extensions of the form
\[0\to V\to E\to \Z(0)\to0\]
for $V$ in $\cM_{-1}$.  In this case, $\mathcal{M}_{ 0}$ is isomorphic to the intermediate Jacobian
\[J(\cV):=\cV_\C/F^0\cV+\cV_\Z\] 
over $\cM_{-1}$ where $\cV_\Z$ is the universal $\Z$-local system and $F^\bullet\cV$ is the universal Hodge filtration.  Moreover, we have maps
\[J(W_{-2}\cV)\to J(\cV)\xrightarrow{\pi} J(\gr^W_{-1}\cV)\]
of definable analytic spaces by interpreting each as mixed period spaces.  In particular, $\cM_{-1}\times_{\cD_{-1}}\cM_{[-1,0]}$ is isomorphic to $J(\gr^W_{-1}\cV)$ over $\cM_{-1}$ and $\pi$ has a definable action by $J(W_{-2}\cV)$ which around every point of $J(\gr^W_{-1}\cV)$ admits an analytic (hence definable, after shrinking) trivializing section.


%

For any $y\in J(\gr_{-1}^W\cV)$, let $U'\subset J(\gr_{-1}^W\cV)$ be a small ball neighborhood, and for $z\in \cM_{-1}$ the image of $y$ let $U\subset\cM_{-1}$ be the (open) image of $U'$.  Denote  \[J^{\Hdg}_z(W_{-2}\cV_U):=(W_{-2}\cV_U)_{\C}/F^0W_{-2}\cV_U+\Hdg_{-1}(W_{-2}V_z)_\Z,\]
where $\Hdg_{k}(H)_\Z:=F^{k}H\cap W_{2k}H_\Z$.  We have a diagram
\[J^{\Hdg}_z(W_{-2}\cV_U)\to J(W_{-2}\cV_U)\to U.\]
Note that independently of the choice of a basepoint there is an identification of the fundamental group of $J(W_{-2}\cV_U)$ with $W_{-2}V_\Z$, that of $J^{\Hdg}_z(W_{-2}\cV_U)$ with $\Hdg_{-1}(W_{-2}V_z)_\Z$, and that the first map is a covering map with covering group $V_\Z/\Hdg_{-1}(W_{-2}V_z)_\Z$.

We endow $J^{\Hdg}_z(W_{-2}\cV_U)$ with a definable structure as follows.  We may choose a definable analytic splitting\[W_{-2}\cV_U\otimes \O_U=F^0W_{-2}\cV_U\oplus (\Hdg_{-1}(W_{-2}V_z)_\Z\otimes\O_U)\oplus Q\]
where $\Hdg_{-1}(W_{-2}V_z)_\Z\otimes\O_U\subset\cV_U\otimes\O_U$ is the constant subbundle and $Q$ is a definable analytic subbundle.  We may therefore definably identify $$J_z^{\Hdg}(W_{-2}\cV_U)\cong \Hdg_{-1}(W_{-2}V_z)_{\G_m}\times \mathbb{C}^N\times U$$ over $U$, where for a mixed Hodge structure $H$ $$\Hdg_{-1}(H)_{\G_m}:=(\Hdg_{-1}(H)_{\Z}\otimes\C)/\Hdg_{-1}(H)_{\Z}\cong\G_m^n$$ with its canonical definable structure.

Choosing a definable section of $\pi$, we identify
\[\pi^{-1}(U')\cong J(W_{-2}\cV_U)\times_U U'.\]
The fundamental group of $\pi^{-1}(U')$ (after choosing a basepoint) is canonically identified with $(W_{-2}V_z)_\Z$ via the action of $J(W_{-2}\cV_U)$.  Denote by 
\[J^{\Hdg}_{U',y}:=J^{\Hdg}_z(W_{-2}\cV_U)\times_U U'\]
and note that $J^{\Hdg}_{U',y}$ and its definable structure don't depend on the choices.  Furthermore, since $J_z^{\Hdg}(W_{-2}\cV_U)\to U$ is clearly definable-analytically quasiaffine (even affine), we have:
\begin{lemma}\label{lemma def an affine}$J^{\Hdg}_{U',y}\to U'$ is definable-analytically quasiaffine.
\end{lemma}

We are now ready to prove Proposition \ref{prop induct step}.
\begin{proof}[Proof of Proposition \ref{prop induct step}]  By \cite[II.4.6.16]{egaii} we may assume $X,Y,Z$ are all reduced.  We first show that $\Theta_{[-1,0]}$ is $h$-ample.  $\Theta_{[-1,0]}$ is algebraic by Proposition \ref{prop theta alg gen}, and since $h$ is proper it suffices to show it is ample on fibers, which is Corollary \ref{cor theta ample one}.

It remains to prove that $\O_X$ is $g$-ample.  Using Lemmas \ref{cor qa crit} and \ref{lemma def an affine}, it is sufficient to show the following:
\begin{claim}  For any $y\in Y$ and any sufficiently small open ball neighborhood $y\in U'\subset J(\gr^W_{-1}\cV)$ as above we have a definable analytic lifting 
\begin{equation}\xymatrix{
&J^{\Hdg}_{U',y}\ar[d]\\
X_{U'}\ar@{-->}[ru]\ar[r]&\pi^{-1}(U')
}\label{eq lift}\end{equation}
where $X_{U'}:=X\cap \pi^{-1}(U')$.
\end{claim}
\begin{proof}  While $J^{\Hdg}_{U',y}\to \pi^{-1}(U')$ is of course not definable, we first claim:  

\begin{lemma}Any analytic lift as in \eqref{eq lift} is definable.
\end{lemma}

\begin{proof}Let $\Xi_0\subset M_0$ be a definable fundamental set for $\cM_0$ and let $\Xi'$ be the preimage of $X_{U'}$ in $\Xi_0$.  Its enough to show that $\Xi'\to J^{\Hdg}_{U',y}$ is definable.  Recall that $M_0$ is identified with $V_\C/F^0V$ over $M_{-1}$.  Taking a resolution of $X$, there are finitely many nilpotent orbits approximating the preimage of $X$ in $\Xi_0$; by \cite[Theorem 6.2]{hpmixnil}, outside of a bounded set $\Xi'$ is within a finite distance (with respect to the standard metric on $V_\C/F^0V$) of finitely many nilpotent orbits whose monodromy is trivial in $M_{-1}$.  It thus suffices to verify that each such nilpotent orbit (restricted to a product of bounded vertical strips) has definable image in $J^{\Hdg}_{U',y}$.  But possibly after shrinking $U'$, each such nilpotent is $v+\sum_i t_in_i$ for $v\in V_\C/F^0V$ and $n_i\in \Hdg_{-1}(V_z)_\Z=\Hdg_{-1}(W_{-2}V_z)_\Z$, for which the claim is obvious.
\end{proof}


By Lemma \ref{lemma saito} the monodromy of the extension
\[0\to \cV_X\to\cE_X\to\Z_X\]
restricted to $X_{U'}$ is an element $\xi$ of $(R^1g_*\cV_X)_y$ which is Hodge of weight 0.  We have an exact sequence
\begin{equation}\label{eq 1}(R^1g_*W_{-2}\cV)_y\to (R^1g_*\cV)_y\to (R^1g_*\gr^W_{-1}\cV)_y\end{equation}
and as the extension 
\[0\to \gr^W_{-1}\cV_X\to \gr_{[-1,0]}^W E\to\Z_X\to 0\]
is pulled back from $Y$, $\xi$ maps to $0$ under the right map of \eqref{eq 1}.  Thus, $\xi$ comes from a class of $(R^1g_*W_{-2}\cV)_y$ which is Hodge of weight 0, and by Lemma \ref{lemma numbers} this is an element of
\[\Hom(H_1(X_{U'},\Q),\Hdg_{-1}(W_{-2}V_y)_\Q).\]
\end{proof}
\end{proof}

\subsection{General case}Let $\cM$ be a mixed period space parametrizing graded-polarized integral Hodge structures $E$.  Let $\cD_w$ be the polarized pure period space of the associated graded object $ \gr^W_w E$ and $\cM_{[w-1,w]}$ the graded-polarized mixed period space of $\gr^W_{[w-1,w]}E$.  We have a diagram  
\[\xymatrix{
\cM\ar[dd]\ar[rd]&\\
&\prod_w \cM_{[w-1,w]}\ar[ld]\\
\cD:=\prod_w\cD_w&
}\]
where the bottom diagonal map arises from the fact that the natural map $\prod_w \cM_{[w-1,w]}\to\cD\times\cD$ factors through the diagonal.

Consider a diagram
\[\xymatrix{
X\ar[dd]_{f}\ar[rd]^g\ar[rr]&&\cM\ar[dd]\ar[rd]&\\
&Y\ar[rr]\ar[ld]^h&&\prod_w \cM_{[w-1,w]}\ar[ld]\\
Z\ar[rr]&&\cD&
}\]
where $X,Y,Z$ are algebraic spaces and the horizontal maps  are Griffiths transverse closed immersions.  The proof of the quasiprojectivity claim of Theorem \ref{thm main} is completed by the following:
\begin{lemma}\hspace{1in} 
\begin{enumerate}
\item $\O_X$ is $g$-ample;
\item $\Theta_X:=\bigotimes_w\Theta_{[w-1,w]}$ is $h$-ample.
\end{enumerate}
\end{lemma}
\begin{proof}
\begin{enumerate}
\item

  If $w_{\min}$ (resp. $w_{\max}$) is the minimum (resp. maximum) weight $w$ for which $\gr_w^WE\neq 0$, then by taking images we have diagrams
\[\xymatrix{
X_w\ar[r]\ar[d]_{f_w}&\cM_{w}\ar[d]\ar[r]&\cM_{[w-1,w]}\\
X_{w-1}\ar[r]&\cM_{w-1}&
}\]
 with $X=X_{w_{\max}}$, $\cM= \cM_{w_{\max}}$, $Z=X_{w_{\min}}$, and $\cD= \cM_{w_{\min}}$.  By Corollary \ref{cor theta ample} the theta bundle $\Theta_{[w-1,w]}$ is $f_w$-ample, and it follows that $L:=\bigoplus_w \Theta_{[w-1,w]}^{a_w}$ is $f$-ample for some $a_w>0$ \cite[II.4.6.13]{egaii}.  As $L$ is pulled back from $Y$, it follows that $\O_X$ is $g$-ample.
 \item  $\Theta_X$ is ample on the fibers of $h$ by Corollary \ref{cor theta ample one}, and since $h$ is proper, it follows that it is $h$-ample.  
\end{enumerate}
\end{proof}

\bibliography{biblio.mixed.griffiths}
\bibliographystyle{plain}
\end{document}